\documentclass[11pt]{article}

\usepackage[margin=1in]{geometry}
\usepackage{amsmath,amsthm,amssymb}
\usepackage{mathtools}
\usepackage{bbm}
\usepackage{graphicx}
\usepackage{booktabs}
\usepackage{hyperref}
\usepackage{enumitem}
\usepackage{algorithm}
\usepackage{algpseudocode}

\hypersetup{
    colorlinks=true,
    linkcolor=blue,
    citecolor=blue,
    urlcolor=blue
}

\newtheorem{theorem}{Theorem}
\newtheorem{lemma}{Lemma}
\newtheorem{proposition}{Proposition}

\theoremstyle{definition}
\newtheorem{definition}{Definition}

\theoremstyle{remark}

\newcommand{\cost}{\mathrm{cost}}
\newcommand{\OPT}{\mathrm{OPT}}
\newcommand{\FF}{\mathrm{FF}}

\title{Approximating k-Center via Farthest-First on $\delta$-Covers}

\author{
    Jason R. Wilson \\
    Department of Mathematics \\
    Virginia Tech \\
    \texttt{jasonwil@vt.edu}
}
\date{\today}

\begin{document}

\maketitle

\begin{abstract}
The farthest-first traversal of Gonzalez is a classical
$2$-approximation algorithm for solving the $k$-center problem, but its
sequential nature makes it difficult to scale to very large datasets.
In this work we study the effect of running farthest-first on a
$\delta$-cover of the dataset rather than on the full set of points.
A $\delta$-cover provides a compact summary of the data in which every
point lies within distance $\delta$ of some selected center.
We prove that if farthest-first is applied to a $\delta$-cover,
the resulting $k$-center radius is at most twice the optimal radius
plus $\delta$.
In our experiments on large high-dimensional datasets, we show that restricting the input to a
$\delta$-cover dramatically reduces the running time of the
farthest-first traversal while only modestly increasing the
$k$-center radius.
\end{abstract}

\section{Introduction}

The $k$-center problem is a fundamental clustering problem in which
a set $P$ of points in a metric space must be covered by $k$ balls so
as to minimize the maximum distance from a point in $P$ to its nearest
center. The problem is NP-hard even in the Euclidean plane
\cite{fowler1981optimal,megiddo1984complexity}.

A simple greedy algorithm due to Gonzalez selects centers iteratively
by repeatedly choosing the point farthest from the current set of
centers~\cite{gonzalez1985clustering}. This farthest-first traversal is
closely related to earlier farthest-point heuristics studied by
Rosenkrantz, Stearns, and Lewis~\cite{Rosenkrantz1977}. The algorithm
yields a $2$-approximation for the $k$-center problem
\cite{gonzalez1985clustering,hochbaum1985best}.

Although the farthest-first traversal provides a strong theoretical
guarantee, it is inherently sequential and typically requires scanning
much of the dataset at each iteration, especially in high dimensions
where geometric pruning becomes ineffective. These properties make the
algorithm difficult to parallelize effectively on large datasets.

In this work, we address these scaling concerns by running
farthest-first on a significantly smaller dataset obtained via a
$\delta$-cover of the data. A $\delta$-cover is a subset of the dataset, whose elements are called centers, such that every point lies within distance $\delta$ of
some center. Informally, the $\delta$-cover provides a geometric
summary of the dataset that can be used to significantly accelerate
downstream tasks.

This raises a natural question: what approximation guarantee can be
obtained if farthest-first is run on a $\delta$-cover instead of the
full dataset? We show that the resulting $k$-center radius is at most
twice the optimal radius plus $\delta$, demonstrating that
farthest-first can be applied to a compact summary of the data while
largely preserving its classical approximation guarantee.

Empirically, we evaluate the approach on two large high-dimensional
datasets: MNIST (1 million points with 100-dimensional PCA features)
and SUSY (5 million points with 18 dimensions). In these experiments,
the $\delta$-cover retains only two percent of the original points,
dramatically reducing the input size. 
Running on the reduced set yields about $50\times$
speedup in the farthest-first traversal itself, with only modest
increases in the $k$-center radius.

In our experiments, the $\delta$-cover is constructed using a simple
uncovered-first strategy: we scan the dataset sequentially and add a
point to the cover whenever it lies farther than $\delta$ from all
previously selected centers. In low-dimensional Euclidean spaces,
$\delta$-covers can be constructed in $O(n)$ time using grid-based
methods~\cite{Dumitrescu2020,Friederich2023}. Scalable construction of
$\delta$-covers in high dimensions is beyond the scope of this
work.

\section{Related Work}

The idea of solving clustering problems on a reduced subset of the
data is not new.  In particular, the notion of a \emph{coreset}~\cite{HarPeled2004} has
been widely used to accelerate clustering algorithms by reducing the
size of the input while approximately preserving the clustering
objective.  
Our work differs from the coreset 
approach in that we do not attempt to construct a subset that
directly approximates the clustering objective.  Instead, we study the
effect of restricting farthest-first traversal to a 
$\delta$-cover of the dataset and analyze how this restriction affects
the classical approximation guarantee.

The construction of covers of point sets has been studied extensively
in computational geometry. For example, Dumitrescu et
al.~\cite{Dumitrescu2020} analyze online algorithms for unit covering
in Euclidean spaces, and Friederich et al.~\cite{Friederich2023}
experimentally evaluate algorithms for covering massive point sets in two
dimensions. Their results indicate that such covers can often be constructed very
efficiently, particularly in low-dimensional settings. This evidence
supports the practical use of $\delta$-covers to accelerate downstream tasks
such as farthest-first traversal.

\section{Preliminaries}

In this section we establish our key notation and review farthest-first traversal as well as its classical approximation result for the $k$-center problem.

\subsection{Notation}

Let $(P,d)$ be a finite metric space.
Throughout the paper, all subsets are taken with respect to $P$.

\vspace{2pt}
For $s \in P$ and $C \subseteq P$, define
\[
d(s,C)
:=
\min_{ c \in C} d(s,c)
\]

\vspace{2pt}
For $C \subseteq P$ and $S \subseteq P$, define
\[
\cost(C,S)
=
\max_{ s \in S} d(s,C)
\]

\vspace{2pt}
For a fixed integer $k \ge 1$, define
\[
\OPT
=
\min_{ \substack{C \subseteq P \\ |C| = k}}
\cost(C,P)
\]

\subsection{Farthest-First Traversal}

We recall the classical farthest-first traversal of Gonzalez for the
$k$-center problem (see Algorithm~1).

\begin{algorithm}
\caption{Farthest-First on $S$}
\begin{algorithmic}[1]
\State \textbf{Input:} $S \subseteq P$, integer $k \ge 1$
\State \textbf{Output:} $C \subseteq S$ with $|C| = k$

\State Choose $c_1 \in S$ arbitrarily
\State $C \gets \{c_1\}$

\For{$i = 2$ to $k$}
    \State $c_i \gets \arg\max_{s \in S} d(s,C)$
    \State $C \gets C \cup \{c_i\}$
\EndFor

\State \Return $C$
\end{algorithmic}
\end{algorithm}

For any $S \subseteq P$, we write $\FF(S)$ to denote
any set of centers obtainable by a valid execution
of Algorithm~1 on $S$.
We emphasize that $\FF(S)$ is not a single deterministic
set: different choices of the initial center and
tie-breaking may produce different outputs.
All results in this paper hold uniformly over
all valid farthest-first executions.

\subsection{Classical Approximation Guarantee}

We recall the classical guarantee for farthest-first
when applied to the full dataset $P$.

\begin{theorem}[Gonzalez, 1985]
For every execution of Algorithm~1 on $P$,
\[
\cost(\FF(P),P) \le 2\,\OPT
\]
\end{theorem}

The inequality states that if the algorithm is run on $P$,
producing a set of centers $C \subseteq P$ with $|C|=k$,
then every point of $P$ lies within distance at most
$2\,\OPT$ of some center in $C$.

\begin{proof}
Let $C = \FF(P)$ and define
\[
r := \cost(C,P) = \max_{p \in P} d(p,C)
\]
Let $x \in P$ satisfy $d(x,C) = r$.

Since $r = d(x,C)$, we have $d(x,c) \ge r$ for every $c \in C$.
At each step the algorithm selects a point maximizing its distance
from the current set of centers. Because $x$ was never selected,
every selected center had distance at least $r$ from the previously
chosen centers. Thus every pair of centers in $C$ has distance
at least $r$.

Thus the set $C \cup \{x\}$ consists of $k+1$
points that are pairwise at distance at least $r$.

Let $C^\star$ be an optimal solution with $|C^\star| = k$
and $\cost(C^\star,P) = \OPT$.
Since $|C^\star| = k$ but $|C \cup \{x\}| = k+1$,
by the pigeonhole principle there exist two distinct points
$u,v \in C \cup \{x\}$ that are assigned to the same
optimal center $c^\star \in C^\star$.

Then
\[
d(u,c^\star) \le \OPT,
\qquad
d(v,c^\star) \le \OPT
\]
By the triangle inequality,
\[
d(u,v) \le 2\,\OPT
\]
Since $d(u,v) \ge r$, we conclude that
\[
r \le 2\,\OPT
\]

Therefore $\cost(C,P) = r \le 2\,\OPT$.
\end{proof}

The factor $2$ in the preceding theorem is tight.
Consider the metric space
\[
P = \{0,1,2\}
\quad\text{with}\quad
d(x,y)=|x-y|
\]
and let $k=1$.

The optimal solution contains the center $1$,
yielding $\OPT = 1$.
However, if the farthest-first algorithm selects $0$ as its
initial center, then
\[
\cost(\FF(P),P) = 2
\]
Thus $\cost(\FF(P),P) = 2\,\OPT$,
showing that the factor $2$ cannot be improved.

\section{Restricting Farthest-First to Geometric Covers}

We now consider running farthest-first on subsets $Q \subseteq P$ that approximate $P$
in a geometric sense.

\begin{definition}[$\delta$-Cover]
Let $\delta > 0$.  A subset $Q \subseteq P$ is a $\delta$-cover of $P$
if for every $p \in P$ there exists $q \in Q$
such that
\[
d(p,q) \le \delta
\]
\end{definition}

Suppose $Q$ is a $\delta$-cover of $P$.
If we apply farthest-first to $Q$ rather than to $P$,
how does $\cost(\FF(Q),P)$ compare to $\OPT$?

For the full dataset $P$, farthest-first satisfies
\[
\cost(\FF(P),P) \le 2\,\OPT
\]
and this bound is tight.

When $Q=P$, we have $\delta=0$, and the classical example with
$P = \{0,1,2\}$ shows that the factor $2$ in front of $\OPT$
cannot be improved.
It is therefore natural to ask whether a universal bound of the form
\[
\cost(\FF(Q),P)
\le
2\,\OPT + \gamma\,\delta
\]
holds for some constant $\gamma \ge 0$.

\begin{proposition}
Suppose there exists a constant $\gamma \ge 0$ such that
\[
\cost(\FF(Q),P)
\le
2\,\OPT + \gamma\,\delta
\]
for every metric space $P$ and every $\delta$-cover $Q \subseteq P$.
Then $\gamma \ge 1$.
\end{proposition}

\begin{proof}
Let $P=\{0,1,\delta+1\}$ with $d(x,y)=|x-y|$ and set $k=2$.
Let $Q=\{0,1\}$.
Then $Q$ is a $\delta$-cover of $P$.
Running farthest-first on $Q$ yields $\FF(Q)=\{0,1\}$,
so $\cost(\FF(Q),P)=\delta$.
An optimal solution for $P$ is $\{1,\delta+1\}$,
giving $\OPT=1$.
The proposed bound would therefore require
\[
\delta \le 2 + \gamma\,\delta
\]
or equivalently $(1-\gamma)\delta \le 2$.
If $\gamma<1$, this fails for sufficiently large $\delta$.
Hence $\gamma \ge 1$.
\end{proof}

The preceding proposition establishes a necessary condition:
any bound of the proposed form must satisfy $\gamma \ge 1$.
It is therefore natural to ask whether the sharp bound
\begin{equation} \label{main_result}
\cost(\FF(Q),P)
\le
2\,\OPT + \delta
\end{equation}
actually holds.  Our main result is that \eqref{main_result} does indeed hold.

\subsection{A Failed Attempt}

A first attempt at showing \eqref{main_result}
is to compare $\cost(\FF(Q),P)$ directly to $\cost(\FF(P),P)$.
Since $Q$ is a $\delta$-cover of $P$,
one might hope that
\begin{equation} \label{false_inequality}
\cost(\FF(Q),P)
\le
\cost(\FF(P),P) + \delta
\end{equation}
If \eqref{false_inequality} were true, 
the classical guarantee
$\cost(\FF(P),P) \le 2\,\OPT$
would immediately imply the desired bound \eqref{main_result}.  

Indeed, restricting farthest-first to a $\delta$-cover
can sometimes reduce the cost.

Consider $P=\{0,2,3,4\}$ with $d(x,y)=|x-y|$
and $k=2$.
Let $Q=\{0,3\}$, which is a $\delta$-cover of $P$ with $\delta = 1$.

Running farthest-first on $P$ starting at $0$
yields $\FF(P)=\{0,4\}$ and
\[
\cost(\FF(P),P)=2.
\]
On $Q$, we obtain $\FF(Q)=\{0,3\}$ and
\[
\cost(\FF(Q),P)=1.
\]

This example shows that restriction may decrease the covering cost,
since it can prevent farthest-first from selecting an unfavorable center.
However, the behavior of farthest-first
under restriction is not monotone in general.

The following two-dimensional example
demonstrates that \eqref{false_inequality}
does not hold in general.

Let
\[
P=\{(0,0),(-2,3),(0,3),(2,3),(0,4)\}
\]
with Euclidean distance and define
\[
Q=\{(0,0),(-2,3),(0,3),(2,3)\}
\]
Then $Q$ is a $\delta$-cover of $P$ with $\delta = 1$.

Let $k = 2$.  
Starting at $(0,0)$, farthest-first on $P$
selects $(0,4)$ as the second center, yielding
\[
\FF(P)=\{(0,0),(0,4)\}
\quad\text{and}\quad
\cost(\FF(P),P)=\sqrt{5}
\]
On $Q$, breaking ties arbitrarily,
the algorithm may select $(-2,3)$, giving
\[
\FF(Q)=\{(0,0),(-2,3)\}
\quad\text{and}\quad
\cost(\FF(Q),P)=\sqrt{13}
\]
Since $\sqrt{13} > \sqrt{5}+1$,
the inequality \eqref{false_inequality} fails.

\section{The Lifting Lemma}

Our first attempt to understand how running farthest-first on the
$\delta$-cover $Q$ affects the covering cost was to compare the
coverings $\FF(Q)$ and $\FF(P)$ while keeping the dataset $P$ fixed.
The preceding example shows that this approach is too ambitious:
the two coverings can differ substantially, and their costs on $P$
cannot be related by such a simple $\delta$-additive bound.

Instead, we fix the covering and vary the dataset.
In this setting the triangle inequality provides a clean way to
reason about the effect of $\delta$.
In particular, it yields a simple relationship between the cost
of a fixed center set $C$ when evaluated on $Q$ and when evaluated
on $P$.

\begin{lemma}[Lifting Lemma]
Let $Q \subseteq P$ be a $\delta$-cover of $P$.
For every $C \subseteq Q$,
\[
\cost(C,P)
\le
\cost(C,Q) + \delta
\]
\end{lemma}

\begin{proof}
Let $p \in P$ satisfy $d(p,C)=\cost(C,P)$.
Since $Q$ is a $\delta$-cover of $P$,
there exists $q \in Q$ with $d(p,q) \le \delta$.
By the triangle inequality,
\[
d(p,C)
\le
d(p,q) + d(q,C)
\le
\delta + \cost(C,Q)
\]
Since $d(p,C)=\cost(C,P)$,
the result follows.
\end{proof}

Applying the lifting lemma with $C=\FF(Q)$ gives
\begin{equation} \label{lifting_cor}
\cost(\FF(Q),P)
\le
\cost(\FF(Q),Q) + \delta
\end{equation}

Thus, to establish \eqref{main_result}, it suffices to control
$\cost(\FF(Q),Q)$ when $Q$ is a $\delta$-cover.

\subsection{A Second Failed Attempt}

A natural next step is to ask whether
\begin{equation} \label{false_inequality_2}
\cost(\FF(Q),Q)
\le
\cost(\FF(P),P)
\end{equation}
holds.
If \eqref{false_inequality_2} held, the classical guarantee
$\cost(\FF(P),P)\le 2\,\OPT$
would imply
\[
\cost(\FF(Q),Q)\le 2\,\OPT
\]
and substituting this bound into \eqref{lifting_cor} would yield
\eqref{main_result}.

However, the inequality
\[
\cost(\FF(Q),Q)
\le
\cost(\FF(P),P)
\]
does not hold in general.

Let
\[
P=\{(0,0),(-2,3),(2,3),(0,4)\}
\]
with Euclidean distance, and define
\[
Q=\{(0,0),(-2,3),(2,3)\}
\]
Then $Q$ is a $\delta$-cover of $P$ with $\delta = \sqrt{5}$.

Let $k=2$.
Starting at $(0,0)$, farthest-first on $P$ selects $(0,4)$
as the second center, yielding
\[
\cost(\FF(P),P)=\sqrt{5}
\]
On $Q$, the algorithm may select $(-2,3)$ as the second center,
giving
\[
\cost(\FF(Q),Q)=\sqrt{13}
\]
Since $\sqrt{13}>\sqrt{5}$, the inequality \eqref{false_inequality_2} fails.

\section{The Main Result}

Two straightforward attempts to prove \eqref{main_result} failed because
they attempted to compare the greedy runs $\FF(Q)$ and $\FF(P)$ directly.
The key insight that leads to our main result is that this comparison is unnecessary:
it is much easier to compare $\FF(Q)$ to an optimal solution.

In fact, the classical packing argument of Gonzalez already provides
exactly the bound we need.  
Remarkably, that packing argument applies without modification when farthest-first is run on any subset of $P$.

\begin{lemma}[Subset Lemma]
Let $S \subseteq P$ with $|S| \ge k$.
Then
\[
\cost(\FF(S),S) \le 2\,\OPT
\]
\end{lemma}

Note that $\OPT$ denotes the optimal $k$-center cost for the full dataset $P$, not for the subset $S$.

\begin{proof}
Let $C = \FF(S)$ and define
\[
r := \cost(C,S)
\]
Choose $x \in S$ such that $d(x,C)=r$.

Since $x$ was not selected as a center,
the centers in $C$ must be pairwise
at distance at least $r$.
Thus the $k+1$ points in $C \cup \{x\}$
are pairwise at distance at least $r$.

Let $C^\star \subseteq P$ be an optimal solution
with $\cost(C^\star,P)=\OPT$.
Every point of $S \subseteq P$
lies within distance $\OPT$ of some center in $C^\star$.

Since $|C^\star| = k$ but $|C \cup \{x\}| = k+1$,
by the pigeonhole principle there exist
two distinct points $u,v \in C \cup \{x\}$
assigned to the same optimal center $c^\star \in C^\star$.
Then
\[
d(u,c^\star) \le \OPT,
\qquad
d(v,c^\star) \le \OPT
\]
and by the triangle inequality
\[
d(u,v) \le 2\,\OPT
\]

Since $d(u,v) \ge r$, we conclude
\[
r \le 2\,\OPT
\]

Therefore $\cost(C,S) = r \le 2\,\OPT$
\end{proof}

At first glance this lemma may seem unrelated to our goal,
since it bounds the cost of $\FF(S)$ only on the set $S$ itself,
not on $P$. However, when $S$ is a $\delta$-cover of $P$,
the lifting lemma allows us to transfer this bound
to the cost on the full dataset with the addition of the $\delta$ term.

\begin{theorem}
Let $Q \subseteq P$ be a $\delta$-cover of $P$
with $|Q| \ge k$.
Then
\[
\cost(\FF(Q),P)
\le
2\,\OPT + \delta
\]
\end{theorem}

\begin{proof}
By the subset lemma,
\[
\cost(\FF(Q),Q) \le 2\,\OPT
\]

Applying the lifting lemma with $C=\FF(Q)$ gives
\[
\cost(\FF(Q),P)
\le
\cost(\FF(Q),Q) + \delta
\]

Combining these inequalities yields
\[
\cost(\FF(Q),P) \le 2\,\OPT + \delta
\]
\end{proof}

\section{Empirical Illustration}

In this section we compare the $k$-center radii 
$\cost(\FF(P),P)$ and $\cost(\FF(Q),P)$,
where $Q$ is a $\delta$-cover of $P$.
Our goal is to observe how the covering radius changes
when farthest-first is applied to the reduced dataset $Q$.

For each dataset, the $\delta$-cover was found using the
uncovered-first algorithm (see Algorithm~2).

\begin{algorithm}[H]
\caption{Uncovered-First $\delta$-Cover Construction}
\begin{algorithmic}[1]
\State \textbf{Input:} $P=\{x_1,\dots,x_n\}\subseteq \mathbb{R}^d$, radius $\delta>0$
\State \textbf{Output:} $\delta$-cover $Q \subseteq P$

\State $Q \gets \{x_1\}$

\For{$i=2$ to $n$}
    \If{$d(x_i,Q) > \delta$}
        \State $Q \gets Q \cup \{x_i\}$
    \EndIf
\EndFor

\State \Return $Q$
\end{algorithmic}
\end{algorithm}

We tested on two large datasets.
The SUSY dataset contains $5{,}000{,}000$ points (18 dimensions, $z$-transformed),
and MNIST contains $1{,}000{,}000$ points (100-dimensional PCA).
All computations were performed in single precision.
Since both datasets are high-dimensional, our farthest-first
implementation did not use geometric acceleration.

\begin{center}
\begin{tabular}{lcccccccc}
\toprule
Dataset & $|P|$ & $\delta$ & $|Q|$ & $k$ &
$\cost(\FF(P),P)$ & $\cost(\FF(Q),P)$ &
Ratio & FF Speedup \\
\midrule
SUSY & 5M & 2.0 & 94{,}035 & 1000 &
5.766 & 5.922 & 1.027 & 56$\times$ \\
 &  &  &  & 5000 &
4.049 & 4.335 & 1.071 & 57$\times$ \\
MNIST & 1M & 1500 & 19{,}496 & 1000 &
2039.397 & 2142.366 & 1.051 & 50$\times$ \\
 &  &  &  & 5000 &
1759.204 & 1947.878 & 1.107 & 51$\times$ \\
\bottomrule
\end{tabular}
\end{center}

Here “Ratio” denotes
$\cost(\FF(Q),P) / \cost(\FF(P),P)$.
In both datasets the $\delta$-cover retained roughly $2\%$
of the original points.
Across both $k$ values the increase in covering radius
remained modest (2.7\%–10.7\%),
while the farthest-first traversal itself was about 50 times faster.

In these experiments we observed that
\[
\left|\cost(\FF(Q),P) - \cost(\FF(P),P)\right|
\]
was much smaller than $\delta$.
Although this behavior is not guaranteed by our theory,
it is encouraging to observe in practice.

\section{Conclusion}

In this paper we studied the behavior of farthest-first traversal
when applied to a $\delta$-cover of the dataset.
Although restricting the greedy traversal to a reduced set of points
can substantially accelerate the computation,
it is not immediately clear how this restriction affects the well-known
$2\,\OPT$ bound for the $k$-center problem on the full dataset.

Our main result shows that if $Q$ is a $\delta$-cover of $P$, then
\[
\cost(\FF(Q),P) \le 2\,\OPT + \delta
\]
The key observation is that the classical packing argument of
Gonzalez extends directly to farthest-first traversal on any subset
of the dataset.
Combining this observation with a simple triangle inequality argument
yields the desired bound.

Empirically, we observe that restricting farthest-first to a
$\delta$-cover can substantially accelerate the farthest-first traversal
while only modestly increasing the $k$-center radius.
In our experiments the $\delta$-cover retained roughly two percent
of the original points, resulting in about $50\times$ speedups
for the farthest-first traversal,
with only 2.7\%–10.7\% increases in the resulting $k$-center radius
depending on $k$.

One important issue not addressed in this work is the fast construction of
the $\delta$-cover itself.
In our experiments we used a simple uncovered-first algorithm,
but scalable methods for constructing $\delta$-covers
are important for making this approach to accelerating $k$-center practical.
Although grid-based binning methods enable $O(n)$ $\delta$-cover
construction in low dimensions, fast $\delta$-cover algorithms
for high-dimensional datasets remain an important direction
for future work.

Finally, an interesting open question is whether the degradation relative
to running farthest-first on the full dataset can be bounded directly
in terms of $\delta$. Although we showed that a simple additive $\delta$
bound does not hold in general, it is possible that some other bound
involving $\delta$ may exist (perhaps depending on dimension).
Since our experiments suggest that the difference between
$\cost(\FF(Q),P)$ and $\cost(\FF(P),P)$ is often much smaller than $\delta$
in practice, a probabilistic analysis of the expected cost difference
may also prove fruitful.

\bibliographystyle{plain}
\bibliography{references}

\end{document}